\documentclass{amsart}
\usepackage{verbatim,amsmath,amssymb}

\newcommand{\bin}[2]{\left( \begin{array}{c} #1\\ #2 \end{array} \right)}

\newcommand{\C}{\mathcal{C}}
\newcommand{\classes}{\mathrm{classes}}
\newcommand{\Des}{\mathrm{Des}}
\newcommand{\inv}{\mathrm{inv}}
\newcommand{\maj}{\mathrm{maj}}
\newcommand{\Lex}[1]{\mathrm{\textbf{Lex}}(#1)}
\newcommand{\M}{\mathcal{M}}
\newcommand{\Recoil}{\mathrm{Recoil}}
\newcommand{\sign}{\mathrm{sign}}
\newcommand{\rev}{\mathrm{rev}}

\def\FTT{\Phi}
\def\NFT{\mathrm{NS}}
\def\QSym{{\it QSym}}          
\def\FQSym{{\bf FQSym}}        

\numberwithin{equation}{section}

\newtheorem{theorem}{Theorem}
\newtheorem{proposition}{Proposition}
\newtheorem{lemma}{Lemma}
\newtheorem{corollary}{Corollary}

\theoremstyle{definition}

\newtheorem{remark}{Remark}
\newtheorem{example}{Example}

\begin{document}

\title{The forgotten monoid}

\author[J.-C. Novelli]{Jean-Christophe Novelli}
\address{Institut Gaspard Monge, Universit\'e Paris-Est,
5 Boulevard Descartes, Champs-sur-Marne, 77454 Marne-la-Vall\'ee cedex 2,
France}
\email{novelli@univ-mlv.fr}
\urladdr{http://www-igm.univ-mlv.fr/\~{}novelli}

\author[A. Schilling]{Anne Schilling}
\address{Department of Mathematics, University of California, One Shields
Avenue, Davis, CA 95616-8633, U.S.A.}
\email{anne@math.ucdavis.edu}
\urladdr{http://www.math.ucdavis.edu/\~{}anne}
\thanks{\textit{Date:} June 2007}
 
\begin{abstract}
We study properties of the forgotten monoid which appeared in work of
Lascoux and Sch\"utzenberger and recently resurfaced in the construction
of dual equivalence graphs by Assaf. In particular, we provide an
explicit characterization of the forgotten classes in terms of inversion
numbers and show that there are $n^2-3n+4$ forgotten classes in the symmetric
group $S_n$.
Each forgotten class contains a canonical element that can be
characterized by pattern avoidance. We also show that the sum of
Gessel's quasi-symmetric functions over a forgotten class is a 0-1 sum
of ribbon-Schur functions.
\end{abstract}

\maketitle

\section{Introduction}

The forgotten monoid appeared in a paper by Lascoux and 
Sch\"utzenberger~\cite{LS:1981} in their construction of the plactic monoid
and was subsequently forgotten. The current research has 
been motivated by the recent resurface of the forgotten monoid in
the definition of dual equivalence graphs in Assaf's study of LLT
polynomials and Macdonald theory~\cite{Assaf:2007}.

In their seminal paper~\cite{LS:1981}, Lascoux and
Sch\"utzenberger introduced the well-known \emph{plactic monoid} as a monoid
that allows noncommutative symmetric elementary symmetric functions to
commute. Their main goal was to obtain a combinatorial proof that the
product of two Schur functions decomposes as a sum of Schur functions with
nonnegative integer coefficients, also known as the Littlewood--Richardson
coefficients.

Their idea can be rephrased by starting with a \emph{noncommutative} analog of
the algebra of symmetric functions. The usual commutative symmetric functions 
are obtained by imposing certain \emph{minimal commuting conditions} on the
variables. Working with this noncommutative algebra on partially commuting
variables makes the computation of products of functions easier and more
precise.
This approach led to a very simple proof of the Littlewood--Richardson rule
with the help of Hopf algebras~\cite{DHT:2002}.

Elementary symmetric functions generate the algebra of symmetric
functions. Hence, in order to obtain a partially commuting algebra isomorphic
to its commutative version, one only needs to define noncommutative analogues
of the elementary symmetric functions and ask that their products commute. The
usual definition is as follows: given a (possibly infinite) ordered alphabet
$A=\{a_1<\dots<a_n\}$, let
\begin{equation}
e_k(A) := \sum_{n\geq i_1>\dots>i_k\geq1} a_{i_1} a_{i_2} \dots a_{i_k}.
\end{equation}
The first nontrivial condition that is imposed by commutativity comes from the
relation $e_1e_2=e_2e_1$.
Since this relation needs to hold for an alphabet with two letters, it
requires
\begin{equation}
\label{leteg}
(a_1+a_2) a_2 a_1 = a_2 a_1 (a_1+a_2).
\end{equation}
But since commutativity is also required to be true for an alphabet with three
letters, the relation
\begin{equation}
a_2 a_3a_1 + a_1 a_3a_2 = a_3a_1 a_2 + a_2a_1 a_3,
\end{equation}
must hold. All the other terms of the expansion simplify, either by appearing
on both sides of the equality, or thanks to relation~\eqref{leteg}.

At this point, Lascoux and Sch\"utzenberger make the choice of what are now
known as the \emph{Knuth relations}, introduced by Knuth~\cite{Knuth:1970}
in his study of the Schensted algorithm~\cite{Schensted:1961}:
\begin{equation}
\left\{
\begin{array}{llr}
acb &\equiv cab & \text{for all $a\leq b<c$,} \\
bac &\equiv bca & \text{for all $a<b\leq c$.}
\end{array}
\right.
\end{equation}
It is then easy to show that the quotient of the noncommutative algebra
generated by the $e_k$ by the Knuth (or, plactic) relations on the variables
makes the $e_k$ commute with one another, and hence provides an algebra
isomorphic to the usual commutative symmetric functions.
These relations have been weakened later by Fomin and Greene~\cite{FG:1998}.

But, as we shall see in the sequel, there is another choice of relations that
imposes commutativity of the $e_k$:
\begin{equation}
\label{eq:forgottengen}
\left\{
\begin{array}{rll}
aba &\equiv baa & \text{for all $a<b$,} \\
bab &\equiv bba & \text{for all $a<b$,} \\
acb &\equiv bac & \text{for all $a<b<c$,} \\
bca &\equiv cab & \text{for all $a<b<c$.}
\end{array}
\right.
\end{equation}

To the best of our knowledge, these relations have not been studied in
the literature and hence are called the \emph{forgotten relations} in the
folklore. Considering these relations as equivalence relations on words, it is then
legitimate to make use of the term \emph{forgotten equivalence classes}, or
forgotten classes for short.
In the sequel, we shall generally consider the forgotten classes on
permutations, that is, the classes generated only by the two relations:
\begin{equation}
\label{eq:forgotten}
\left\{
\begin{array}{rll}
acb &\equiv bac & \text{for all $a<b<c$,} \\
bca &\equiv cab & \text{for all $a<b<c$.}
\end{array}
\right.
\end{equation}

For example, here is a complete forgotten class at $n=5$:
\begin{equation}
\begin{split}
\{ & 12543, 13452, 13524, 14253, 14325, 15234, 21453,\\
   & 21534, 23154, 23415, 24135, 31254, 31425, 32145, 41235\}.
\end{split}
\end{equation}

Since the forgotten relations provide some rewritings on consecutive letters,
the quotient of the free algebra on $A$ has naturally a structure of monoid
which we shall call the \emph{forgotten monoid}.

Our initial motivation to study the forgotten monoid is the appearance of the
forgotten relations in the construction of dual equivalence graphs by
Assaf~\cite{Assaf:2007}, based on ideas of Haiman~\cite{Haiman:1992}. Assaf
considers two sets of relations on permutations:
\begin{equation}
\left\{
\begin{array}{rl}
\dots i-1 \dots i+1 \dots i \dots &\equiv^* \dots i \dots i+1 \dots i-1 \dots\\
\dots i \dots i-1 \dots i+1 \dots &\equiv^* \dots i+1 \dots i-1 \dots i \dots
\end{array}
\right.
\end{equation}

\begin{equation}
\label{eq:coforgotten}
\left\{
\begin{array}{rl}
\dots i+1 \dots i-1 \dots i \dots &\equiv^* \dots i \dots i+1 \dots i-1 \dots\\
\dots i \dots i-1 \dots i+1 \dots &\equiv^* \dots i-1 \dots i+1 \dots i \dots
\end{array}
\right.
\end{equation}
The first set consists of the plactic relations on the inverses of the
permutations (also known as the coplactic relations) and the second set
consists of the forgotten relations on the inverse permutation, hence called
the \emph{coforgotten relations}.

In this paper, we show in Theorem~\ref{thm:classes} that the number of
forgotten classes of the symmetric group on $n$ letters $S_n$ is $n^2-3n+4$.
More precisely, two permutations are forgotten-equivalent if and only if they
have the same number of inversions and the letters $1$ and $n$ appear in the
same order in both permutations. In addition, we characterize canonical
elements in each forgotten class by pattern avoidance. It turns out that these
canonical elements are the lexicographically minimal elements in each class.
Theorem~\ref{thm:commut} shows that the quotient of the algebra
generated by the $e_k$ by the forgotten relations is indeed isomorphic to the
algebra of commutative symmetric functions.

Using our characterization of the forgotten classes in terms of inversions we
also prove a conjecture by Zabrocki~\cite{Zabrocki:2007}, independently proved by 
Assaf~\cite{Assaf:2007a}, namely that the sum of Gessel's quasi-symmetric
functions over a forgotten class is 0-1 sum of ribbon Schur functions and
hence is a symmetric function. In addition we give a precise combinatorial
description of the ribbons that appear in the sum. The precise result is stated in
Theorem~\ref{thm:ribbon}, 

The paper is organized as follows. Section~\ref{sec:invariances} provides
\emph{invariants} on the forgotten classes that take at least $n^2-3n+4$
values, hence showing that the number of forgotten classes in $S_n$ is
$\classes(n)\ge n^2-3n+4$.  In Section~\ref{sec:patterns} we provide a set of
so-called \emph{canonical elements} using pattern avoidance and shows that
each permutation can be rewritten as one of these. Since there are $n^2-3n+4$
canonical elements, this shows that $\classes(n)\le n^2-3n+4$. In
Section~\ref{sec:characterization} we state Theorem~\ref{thm:classes} about
the characterization of the forgotten classes in terms of the inversion number
and give many properties which are consequences of this theorem.
In Section~\ref{sec:insertion} an iterative algorithm for the forgotten
relations is presented, analogous to the Schensted insertion algorithm for the
plactic monoid (see~\cite{Schensted:1961} for the original algorithm).
Section~\ref{sec:quotient} uses standard properties to prove
Theorem~\ref{thm:commut}.
The last section is devoted to the proof of Theorem~\ref{thm:ribbon} which
gives the expansion of sums of Gessel's quasi-symmetric functions in terms of
ribbon Schur functions.

\subsection*{Acknowledgments}
The authors would like to thank Sami Assaf, Alain Lascoux, and Mike Zabrocki for interesting
discussions. We would also like to thank Adriano Garsia for careful reading of this manuscript
and pointing out a mistake in one of the arguments in an earlier version of this paper.
JCN was partially supported by the ANR project hopfcombop.
AS was partially supported by NSF grants DMS-0501101, DMS-0652641, and
DMS-0652652.

\section{Invariants}
\label{sec:invariances}

Let $S_n$ denote the set of permutations on $n$ letters. An inversion of a
permutation $\sigma=\sigma_1 \sigma_2 \dots \sigma_n \in S_n$ is a pair
$\{\sigma_i,\sigma_j\}$ such that $\sigma_i>\sigma_j$ but $i<j$. The inversion
number $\inv(\sigma)$ is the number of inversions in $\sigma$.
In this section, we shall only consider the \emph{standard} forgotten classes,
that is the forgotten classes on permutations.

\begin{remark}
\label{rem:inversion}
The forgotten elementary equivalence relations~\eqref{eq:forgotten}
preserve the inversion number of the permutations. Hence the inversion number
is an invariant of each class.
\end{remark}

\begin{remark}
\label{rem:inv}
Directly from the forgotten relations, we have that in each forgotten class of
$S_n$, either $\{1,n\}$ forms an inversion or not.
\end{remark}

\begin{example}
Here is a list of the forgotten classes for $n=2,3,4$ together with the
inversion numbers
\begin{equation}
\begin{array}{|c|l|l|}
\hline 
n & \text{class} &\inv\\ \hline
2 & 12  & 0\\
   & 21 & 1\\ \hline
3 & 123 & 0 \\
   & 132, 213 & 1\\
   & 231, 312 & 2\\
   & 321 & 3\\ \hline
4 & 1234 & 0\\
   & 1243, 1324, 2134 & 1\\
   & 1342, 1423, 2143, 2314, 3124 & 2\\
   & 1432, 3142, 3214 & 3\\
   & 2341, 2413, 4123 & 3\\
   & 2431, 3241, 3412, 4132, 4213 & 4\\
   & 3421, 4231, 4312 & 5\\
   & 4321 & 6\\
\hline
\end{array}
\end{equation}
\end{example}

The classes for which $\{1,n\}$ does not form an inversion are in one-to-one
correspondence with the classes for which $\{1,n\}$ forms an inversion by
the operation $\rev$ reading the word from right to left. It is clear
from~\eqref{eq:forgotten} that $\rev(\sigma)\equiv \rev(\sigma')$ if
$\sigma\equiv \sigma'$.

\begin{proposition}
\label{prop:lower}
The number of forgotten classes satisfies $\classes(n)\ge n^2-3n+4$.
\end{proposition}

\begin{proof}
To get a lower bound on $\classes(n)$, we have to find what number of
inversions have permutations where $1$ is before $n$, and permutations where
$n$ is before~$1$.

The number of inversions of a permutation in $S_n$ can be any number from $0$
to $n(n-1)/2$. If we restrict ourselves to permutations where 1 is before
$n$, then those permutations can have at most $n(n-1)/2-(n-1)=(n-1)(n-2)/2$
inversions since each of the $n-2$ remaining numbers is either to the right of
1 or to the left of $n$, hence having no inversion with at least one of these,
and $\{1,n\}$ is not an inversion.
It is easy to see that $1n(n-1)\dots 2$ has inversion number $(n-1)(n-2)/2$
so that permutations where 1 is before $n$ can have any number of inversions
between $0$ and $(n-1)(n-2)/2$.
This amounts to $1+(n-2)(n-1)/2$ different sets of permutations.

By the symmetry $\rev$ that reads the words from right to left, the same
reasoning works for permutations having $n$ before 1, so that we finally get
$2+(n-2)(n-1)=n^2-3n+4$ sets of permutations such that each forgotten class
belongs entirely to one set.
Hence $\classes(n)\ge (n-1)(n-2)+2 = n^2-3n+4$.
\end{proof}

\section{Canonical elements and pattern avoidance}
\label{sec:patterns}

We consider the set $\Lex{n}$ of permutations defined by pattern avoidance:
they are the elements of $S_n$ 
avoiding the four patterns 213, 312, 13452, and 34521.
A permutation $\sigma=\sigma_1 \sigma_2 \cdots \sigma_n$ avoids the patterns 213 
if there is no subword $\sigma_{i_1} \sigma_{i_2} \sigma_{i_3}$ with $i_1<i_2<i_3$ such
that the relative values of $\sigma_{i_1}$, $\sigma_{i_2}$, and $\sigma_{i_3}$ are the 
same as in 213, that is $\sigma_{i_2}<\sigma_{i_1}<\sigma_{i_3}$. Avoidance of the patterns
312, 13452, and 34521 are defined in same way.

\begin{example}
Here are the elements of $\Lex{n}$ for all $n\leq5$.
\begin{equation}
1;\qquad  12,\ 21; \qquad
123,\ 132,\ 231, \ 321
\end{equation}
\begin{equation}
1234,\ 1243,\ 1342,\ 1432,\ 2341,\ 2431,\ 3421,\ 4321
\end{equation}
\begin{equation}
\begin{split}
&12345,\ 12354,\ 12453,\ 12543,\ 13542,\ 14532,\ 15432, \\
&23451,\ 23541,\ 24531,\ 25431,\ 35421,\ 45321,\ 54321.
\end{split}
\end{equation}
\end{example}

The elements of $\Lex{n}$ can be characterized as follows. Since they
avoid the patterns 213 and 312, they have to be of $\Lambda$-shape, meaning
that they first increase and then decrease.
More precisely
\begin{equation}
\sigma_1 < \sigma _2 < \dots < \sigma_k > \sigma_{k+1} > \dots > \sigma_n
\end{equation}
for some $k$, where $\sigma=\sigma_1\cdots \sigma_n$.
To simplify the presentation, from now on, we will only write the increasing
part of those elements, being understood that the remaining numbers are in
decreasing order to the right of $n$.
Since in addition they avoid the patterns 13452 and 34521 they cannot have a
gap in consecutive numbers in a position greater than 1 followed by three
increasing numbers. Hence the elements of $\Lex{n}$ are of one of the
following forms:
\begin{equation} \label{eq:lex}
\begin{split}
   & 1 2 \dots k \;n \;\\
   & 1 2 \dots k \;a \;n \;\\
   & 2 3 \dots k \;n \;\\
   & 2 3 \dots k \;a \; n \;\\
\end{split}
\end{equation}
where $1\le k<n$ in cases one and three, and $1\le k<n-2$ and $k+1<a<n$ in
cases two and four.

\begin{proposition}
\label{prop:count}
The set $\Lex{n}$ contains $n^2-3n+4$ elements.
\end{proposition}

\begin{proof}
Let us use the characterization in~\eqref{eq:lex}. For case one
in~\eqref{eq:lex}, there are $n-1$ choices for $k$.
For case two in~\eqref{eq:lex}, there are $n-k-2$ choices for each $a$ for a
given $k$, and hence $\sum_{k=1}^{n-2} (n-k-2) = (n-2)^2 - (n-1)(n-2)/2$
choices in total. Hence altogether there are $(n^2-3n+4)/2$ permutations of
form one and two.
The permutations belonging to cases three and four are just the permutations
of form one and two where 1 has been cycled to the end of the word. Hence the
total number of permutations in $\Lex{n}$ is $n^2-3n+4$ as claimed.
\end{proof}

\begin{lemma}
\label{lem:minmax}
Let $F$ be a standard forgotten class of $S_n$. Then $F$ contains either a word 
that begins with $1$, or with $n$. Both cases cannot appear simultaneously inside $F$.
Moreover, if there is a word beginning with $1$ (resp. $n$) in $F$, then there
is a word ending with $n$ (resp. $1$).
\end{lemma}

\begin{proof}
We know that in each forgotten class either $\{1,n\}$ is an inversion or not.
Hence $F$ cannot have both words starting with $1$ and starting with $n$. 
Moreover, the last statement is a consequence of the first statement since the
forgotten elementary rewritings~\eqref{eq:forgotten} are invariant under the
reversal $\rev$ of words.
This shows that every forgotten class must have an element ending with $1$ or
$n$.
Because of the property that in each forgotten class $\{1,n\}$ is either an
inversion or not, the third statement follows.

We now prove the first statement. Let $\sigma$ be a permutation where $1$ is 
before~$n$. 
Write $\sigma$ as $x.w.y$ where $x$ (resp. $y$) is its first (resp. last)
letter, and $w$ the remaining part of $\sigma$. Then, if $y\not=n$, we are 
done by induction on the length of $\sigma$ since the desired property holds 
for $x.w$. If $y=n$, then by induction, depending on the respective order of
$1$ and $n-1$ in $x.w$, the $1$ can be put either to the beginning of $x.w$
(we are then done) or to the end of $x.w$. The last situation means there
would be a word congruent to $\sigma$ that ends with $1n$. But in this case
$1n$ can be moved to the left past all other letters via the forgotten
relations.
Hence in all cases, there exists a word congruent to $w$ that starts with $1$.
\end{proof}

\begin{proposition}
\label{prop:canonical}
The minimal lexicographic element of each forgotten class is an element of
$\Lex{n}$.
\end{proposition}

\begin{proof}
Let $\sigma$ be the minimal lexicographic word of a forgotten class of $S_n$.
We first prove that $\sigma$ avoids the patterns 312 and 213.
Consider a word that contains the patterns 312 or 213 and, among all subwords
having one of these patterns, consider one subword with minimal distance
between the two extreme letters. Let us assume that this subword is a 312
pattern, which we shall write as $x_3x_1x_2$.

There cannot be any letter in $\sigma$ between $x_3$ and $x_1$:
if this letter were greater than $x_2$, we would have a shorter pattern 312,
which contradicts the minimal distance assumption. The letter also cannot be
smaller than $x_1$ for the same reason and it cannot be in the interval of
values between $x_1$ and $x_2$ since we would have a shorter 213 pattern.
Therefore $x_3$ and $x_1$ are consecutive in $\sigma$.
For similar reasons, $x_1$ and $x_2$ are consecutive in $\sigma$. Analogous
arguments work if the minimal subword is a 231 pattern.

This shows that if $\sigma$ contains the pattern 312 or 213, there are three 
consecutive letters in $\sigma$ having one of those patterns. Hence $\sigma$ 
can be rewritten as a lexicographic smaller element via the forgotten
relations~\eqref{eq:forgotten}, which contradicts our assumption that $\sigma$ 
was minimal. Hence $\sigma$ must avoid the patterns 312 and 213 and we 
conclude that it is a $\Lambda$-word.

Let us now write $\sigma$ as $x.w.y$, where $x$ and $y$ are letters and $w$ is
a word. Since $\sigma$ is the smallest element of its forgotten class, $w.y$
and $x.w$ are the smallest elements in their forgotten classes as well. By
induction, they are in $\Lex{n}$ and hence avoid the patterns $13452$ and
$34521$. So if $\sigma$ has one of the two patterns, then both letters $x$ and
$y$ have to be part of it.
Let us first assume that $\sigma$ contains pattern $13452$. Then $x=1$ since 
$x<y$ and $\sigma$ is a $\Lambda$-word. In addition $y=2$, since otherwise the
letter next to $x$ would be smaller than $y$, and $w.y$ would have a 13452 
pattern as well. Also by induction hypothesis, we have that the first letter
of $w$ is $3$.
Putting together everything, we have that $\sigma$ is of the form
$1\, 3\, z\, *\, n\, *\, 2$ where $*$ represents any sequence of numbers,
possibly empty. We claim that there is a word in the forgotten class of
$3\, z\, *\, n\, *\, 2$ of the form 
$a\, b\, n\, 2\, *\,$    with $a<b$.
Since the word $1abn2$ is of the form $13452$, it can be rewritten into
$12nba$, the claim shows that $13z*n*2$ is not minimal in its class,
yielding a contradiction.

Now let us prove the claim. By induction on $n$ of this proposition,
using Corollary~\ref{cor:upper} below and Remarks~\ref{rem:inversion}
and~\ref{rem:inv}, the elements of a forgotten class
are characterized by their inversion number and the property that
either $1$ appears before $n$ or $n$ appears before $1$.
Among all words of the form $a\, b\, n\, 2\, *$, the one with minimal number of
inversions is $3\, 4\, n\, 2\, 5\, 6\dots n-1$  (which has $n-2$ inversions)
and the one with maximal number of inversions is
$n-2\, n-1\, n\, 2\, n-3\, n-4\dots 3$
(which has $\bin{n-2}{2}$ inversions).
Now, among all words of the form $\,a\,b\,n\,2\,*$ with $a<b$, there are permutations
with any number of inversions between those two values.
The number of inversions of $3\,z\,*\,n\,*\,2$ is necessarily inside this
interval. Indeed, among all lexicographically minimal words
of size $n-1$, the one with the smallest number of inversions of
the form $2\, z\, *\, n-1\, *\, 1$  is $2\, 3\, 4\, 5\, 6 ... n-1\, 1$ which has $n-2$
inversions and the one with the greatest number of inversions is
$2\, n-2\, n-1\, n-3 ... 4\, 3\, 1$ which has $\bin{n-2}{2}$ inversions.
Hence $\sigma$ cannot contain the pattern $13452$. A similar
argument works for pattern $34521$.
\end{proof}

\begin{corollary} 
\label{cor:upper}
The number of forgotten classes satisfies $\classes(n)\leq n^2-3n+4$.
\end{corollary}

\begin{proof}
By Proposition~\ref{prop:canonical}, there is at least one permutation of
$\Lex{n}$ in each forgotten class. Combining this result
with Proposition~\ref{prop:count} yields the claim.
\end{proof}

\section{Characterization and properties}
\label{sec:characterization}

In this section we provide an explicit characterization of the forgotten
classes and derive several properties. The main result of this section is the
following theorem.

\begin{theorem}
\label{thm:classes}
The number of forgotten classes of $S_n$ is $\classes(n) = n^2-3n+4$.
More precisely, two permutations are forgotten-equivalent if and only if they
have same number of inversions and the letters $1$ and $n$ appear in the same
order in both permutations.
\end{theorem}

\begin{proof}
The equality $\classes(n)=n^2-3n+4$ is a consequence of
Corollary~\ref{cor:upper} and Proposition~\ref{prop:lower}. The
characterization follows from Remarks~\ref{rem:inversion} and~\ref{rem:inv}.
\end{proof}

\begin{corollary}
\label{cor:inverse}
Let $\sigma$ be in $\Lex{n}$. Then $\sigma^{-1}$ is in the same forgotten
class as $\sigma$, namely $\sigma \equiv \sigma^{-1}$.
\end{corollary}

\begin{proof}
We have that $\inv(\sigma)=\inv(\sigma^{-1})$. In addition, the letters 1 and
$n$ are in the same order in $\sigma$ and $\sigma^{-1}$ for $\sigma\in\Lex{n}$
as can be easily seen from their characterization~\eqref{eq:lex}.
Hence, by Theorem~\ref{thm:classes}, they must be in the same class.
\end{proof}

\begin{corollary}
The elements in $\Lex{n}$ and their inverses constitute canonical elements for
the coforgotten classes corresponding to the relations~\eqref{eq:coforgotten}. 
In addition, the number of standard coforgotten classes is $n^2-3n+4$.
\end{corollary}

The \emph{Sch\"utzenberger involution} $\#$ on a permutation $\sigma\in S_n$
is the composition of the two maps
\begin{itemize}
\item $\rev$ which reverses the order of all letters in $\sigma$;
\item complementation which changes letter $x$ to $n+1-x$ for all letters
$x$ in $\sigma$. 
\end{itemize}
For example,
\begin{equation}
  \# 842956137 = 379451862.
\end{equation}

\begin{corollary}
\label{cor:Schutz}
For any $\sigma\in S_n$, we have $\#\sigma \equiv \sigma$.
\end{corollary}

\begin{proof}
The composition of reversal and complementation does not change the
inversion number of a permutation. In addition, the letters 1 and $n$ stay
in the same order. Hence the corollary follows from Theorem~\ref{thm:classes}.
\end{proof}

Thanks to the characterization of forgotten classes as given in Theorem~\ref{thm:classes}, 
it is easy to find which $\Lambda$-shaped elements belong to the same forgotten class since
the number of inversions of a permutation of $\Lambda$-shape is easy to
compute: it is equal to the maximal number of inversions minus the difference
between $n$ and all the numbers to its left.

Let $\sigma\in S_n$ be a $\Lambda$-shape permutation such that 1 is to the left of $n$.
If there are two letters $(b,c)$ such that $1<b<c<n$ that belong to
the increasing part of $\sigma$ and such that $b-1$ is not in this part,
so that $\sigma$ is not in $\Lex{n}$,
then there is a smaller $\Lambda$-word with the same number of inversions:
look for the smallest number $d$ greater than $c$ that is not in the
increasing part of $\sigma$.
If it exists, then replace $(b,d-1)$ with $(b-1,d)$. Otherwise, replace
$b$ by $b-1$ and remove $n-1$ from the increasing part.
The case when $n$ is to the left of 1 is the same, replacing the
condition $1<b<c<n$ by $2<b<c<n$.

\begin{example} \label{ex-Leqs}
The following example illustrates the last construction.
The permutation 13567842 is forgotten-equivalent to 13468752: in that case,
$b=5$, $c=6$, and $d$ does not exist.
This last permutation is equivalent to 12568743, with $b=3$, $c=4$, and
$d=5$.
This last permutation is equivalent to 12478653, with $b=5$, $c=6$, and
$d=7$.
This last permutation is in turn equivalent to 12387654, with $b=4$, $c=7$, and
$d$ does not exist. All of these permutations have inversion number 10.
\end{example}

\section{The insertion algorithm}
\label{sec:insertion}

Let us now write the elements in~\eqref{eq:lex} more compactly as
\begin{equation}
\begin{split}
  \sigma(k,a) &= 1 2 \ldots k \;a \;n \\
  \tau(k,a) &= 2 3 \ldots k \;a \;n
\end{split}	
\end{equation}
for $1\le k<n$ and $k<a\le n$, where we identify $\sigma(k,n)=\sigma(k-1,k)$
and 
$\tau(k,n)=\tau(k-1,k)$.

\begin{lemma}
\label{lem:inv classes}
We have
\begin{equation}
\begin{split}
	\inv(\sigma(k,a)) & = \bin{n-k}{2}+a-n\\
	\inv(\tau(k,a)) &= \bin{n-k}{2} +a-1.
\end{split}
\end{equation}
\end{lemma}

\begin{proof}
These formulas can easily be checked explicitly.
\end{proof}

Let us define the two subsets of $\Lex{n}$ depending on whether $\{1,n\}$
forms an inversion or not
\begin{equation}
\begin{split}
	\Lex{n,\sigma} &:= \{\sigma(k,a) \mid 1\le k<n, \; k<a\le n\}\\
	\Lex{n,\tau} &:= \{\tau(k,a) \mid 1\le k<n, \; k<a\le n\}.
\end{split}
\end{equation}
By Lemma~\ref{lem:inv classes}, there exist two bijections
\begin{equation}
\begin{split}
 & \inv_\sigma : \Lex{n,\sigma} \to \left\{0,1,\ldots, \bin{n-1}{2} \right\} \\
 & \inv_\tau : \Lex{n,\tau} \to \left\{n-1,n,\ldots, \bin{n}{2} \right\},
\end{split}
\end{equation}
where $\inv_\sigma$ (resp. $\inv_\tau$) is just the inversion number
restricted to the subset of permutations $\Lex{n,\sigma}$ (resp.
$\Lex{n,\tau}$).

The inverses of the maps $\inv_\sigma$ and $\inv_\tau$ can also be stated
explicitly.
Let $i\in \{0,1,\ldots,(n-1)(n-2)/2\}$. Pick the largest integer $m$ such that
$\bin{m}{2}\le i$ and write 
\begin{equation}
	i=\bin{m}{2}+b \quad \text{for $m\ge 1$ and $0\le b<m$}.
\end{equation}
Then $\inv_\sigma^{-1}(i)=\sigma(k,a)$ where $k=n-m-1$ and $a=n+b-m$.

Similarly, if $i\in  \{n-1,n,\ldots,n(n-1)/2\}$, pick $m$ and $b$ as before.
Then $\inv_\tau^{-1}(i)=\tau(k,a)$ where $k=n-m$ and $a=b+1$.

\begin{example}
Take $n=7$ and $i=13$. Then $m=5$ and $b=3$ and
\begin{equation}
 \inv_\sigma^{-1}(13) = 1576542 \quad \text{and} \quad 
 \inv_\tau^{-1}(13) = 2476531.
\end{equation}
\end{example}

The bijections $\inv_\sigma$ and $\inv_\tau$ can be used to define an
insertion algorithm $\omega' \leftarrow i$ of a letter $i\in
\{0,1,\ldots,n-1\}$ into a permutation $\omega' \in \Lex{n-1}$.
Call $\omega\in\Lex{n}$ the result of the insertion algorithm
$\omega' \leftarrow i$. Then
\begin{equation}
\omega = \omega' \leftarrow i = 
\begin{cases}	 
	\inv_\sigma^{-1}(\inv(\omega')+n-1-i) & \text{if $i\neq 0$ and $\omega'=\sigma(k',a')$}\\
	\inv_\tau^{-1}(\inv(\omega')+n-1-i) & \text{if $i=0$ and $\omega'=\sigma(k',a')$}\\
	\inv_\tau^{-1}(\inv(\omega')+n-1-i) & \text{if $i\neq n-1$ and $\omega'=\tau(k',a')$}\\
	\inv_\sigma^{-1}(\inv(\omega')+n-1-i) & \text{if $i=n-1$ and $\omega'=\tau(k',a')$}.
\end{cases}
\end{equation}

\begin{example}
Take $n=7$, $\omega'=\sigma(1,3)=136542$ and $i=0$. Then
$\inv(\omega')+n-1-i=7+7-1=13$, so that $\omega=\inv_\tau^{-1}(13)=2476531$.
More generally, varying $i$ we obtain
\begin{equation}
\begin{array}{|c|l|}
\hline
i & \omega' \leftarrow i\\ \hline
0 & 2476531\\
1 & 1476532\\
2 & 1376542\\
3 & 1276543\\
4 & 1267643\\
5 & 1257643\\
6 & 1247653\\ \hline
\end{array}
\end{equation}
\end{example}

\begin{remark}
\label{prop:insertion}
Let $\omega'\in\Lex{n-1}$ and $i\in \{0,1,\ldots,n-1\}$. Then, by
construction, the standardization of the word $\omega'i$ is in the same
forgotten class as
$\omega'\leftarrow i$.
\end{remark}

\section{Noncommutative forgotten elementary symmetric functions}
\label{sec:quotient}

Let us consider the quotient of the noncommutative elementary symmetric
functions
\begin{equation}
e_k(A) := \sum_{i_1>\dots>i_k\geq1} a_{i_1} a_{i_2} \dots a_{i_k}.
\end{equation}
by the general forgotten relations~\eqref{eq:forgottengen}.

\begin{theorem}
\label{thm:commut}
The quotient of the algebra generated by the $e_k$ by the forgotten relations
is isomorphic to the algebra of commutative symmetric functions.
\end{theorem}

\begin{proof}[Sketch of proof]
As stated in the introduction, it suffices to prove that all $e_k$ commute.
The forgotten relations were chosen such that $e_1$ and $e_2$ commute.

Now, consider the general case: let us evaluate $e_i(A)e_j(A)-e_j(A)e_i(A)$ in
the forgotten monoid.
The noncommutative words appearing in this expression are of two sorts: words
with repetition and words without repetition.

Let us begin with the words without repetition. Since the congruence only
takes into account the mutual ordering of the words, we can restrict ourselves
to permutations. Then there is a simple bijection between the noncommutative
words appearing in $e_ie_j$ and the words appearing in $e_je_i$, that is the
Sch\"utzenberger involution.
By Lemma~\ref{cor:Schutz}, those two words are equivalent and hence their
difference is $0$ in the quotient of the algebra by the forgotten monoid.

Recall that the descent composition of a word is the sequence of lenghts of
all maximal nondecreasing factors of $w$. The property of equivalence of a
permutation and its image by the Sch\"utzenberger involution proves indeed a
more general fact: since the descent composition of $\#\sigma$ is equal to
the reversal of the descent composition of $\sigma$, the set of all
descent composition of a forgotten class is invariant by the operation
consisting in reading compositions from right to left. 
The same property holds on non-standard forgotten classes, so that
the case of words with repetition is solved in the same way.
\end{proof}

\section{Applications to Quasi-symmetric functions}

\def\F{{\bf F}}
\def\R{{\bf R}}

Let
\begin{equation}
	F_{n,D}(x) = \sum_{\substack{i_1\le \cdots \le i_n\\ i_j=i_{j+1} \Rightarrow j\not\in D}}
	 x_{i_1} \cdots x_{i_n}
\end{equation}
be the quasi-symmetric functions first introduced by
Gessel~\cite{Gessel:1984}, defined for any subset $D$ of $\{1,\dots,n-1\}$.
Recall that there is a simple bijection between these subsets and compositions
of $n$ (that is, all positive sequences of sum $n$): send a composition
$I=[i_1,\dots,i_p]$ to $D:=\{i_1,i_1+i_2,\dots,i_1+\dots+i_{p-1}\}$.
We then say that $I$ is the descent composition of $D$.

One of the crucial observations in Assaf's thesis~\cite{Assaf:2007} is the
fact that the sum of Gessel's quasi-symmetric functions for the descents of
the vertices of so-called dual equivalence graphs yield Schur functions. Hence
as a corollary, it follows~\cite[Corollary 3.12]{Assaf:2007} that the
generating function of dual equivalence graphs with a constant statistics on
connected components is a symmetric function and Schur positive.

In this section we prove a similar result conjectured by
Zabrocki~\cite{Zabrocki:2007} and independently proved by 
Assaf~\cite{Assaf:2007a}, stating the expansion of Gessel's quasi-symmetric
functions over a forgotten class as a nonnegative multiplicity-free sum of
skew-ribbon Schur functions. 

Let $\M_{k,+}$ (resp. $\M_{k,-}$) be the set of permutations $\sigma$ such
that $\maj(\sigma^{-1})=k$ and such that $n-1$ comes before (resp. after) $n$.
Here $\maj(\sigma)=\sum_{i\in\Des(\sigma)} i$ is the major index and 
$\Des(\sigma) = \{ i \mid 1\le i<n, \sigma_i>\sigma_{i+1}\}$ is the descent
set of $\sigma$. Let us also define the set $\C_{k,+}$ (resp. $\C_{k,-}$) of
compositions $c=(c_1,\ldots,c_\ell)$ with major index
$\maj(c)=\sum_{i=1}^\ell (\ell-i) c_i=k$ that either end in $1$ (resp. do not
end in $1$).
Then we have 
\begin{equation}
	\M_{k,\pm} = \bigcup_{I\in \C_{k,\pm}} \{\sigma \mid \Recoil(\sigma)=I\},
\end{equation}
where $\Recoil(\sigma)$ is the recoil composition of $\sigma$, or
equivalently the descent composition of $\sigma^{-1}$.

\begin{theorem} \label{thm:ribbon}
Let $S$ be a forgotten class. Then
\begin{equation}
	\sum_{\sigma\in S} F_{n,\Des(\sigma)}(x) = \sum_{I\in \C_{k,\sign(S)}} r_I,
\end{equation}
where $r_I$ is the ribbon Schur function and $\sign(S)=+$ if $1$ comes before
$n$ in $S$ and $\sign(S)=-$ otherwise. In particular, $\sum_{\sigma\in S}
F_{n,\Des(\sigma)}(x)$ is a symmetric function.
\end{theorem}

We provide a combinatorial proof of Theorem~\ref{thm:ribbon}.

\begin{proof}
Let $\FTT$ be the second fundamental transformation of Foata~\cite{F:1968}.
It is well-known that $\FTT$ is a bijection, and that
\begin{equation}
\label{eq0}
	\Recoil(\FTT(\sigma)) = \Recoil(\sigma) \qquad\text{and}\qquad
	\inv \FTT(\sigma) = \maj(\sigma).
\end{equation}
We shall make use of a very trivial fact about $\FTT$:
\begin{equation}
\label{eq1}
	\FTT(\sigma)_1 < \FTT(\sigma)_n \quad \Longleftrightarrow \quad 
\sigma_{n-1}<\sigma_n.
\end{equation}

Let us now consider the map
\begin{equation}
	\NFT (\sigma) := (\FTT(\sigma^{-1}))^{-1}.
\end{equation}
The inverse map interchanges descents and recoils. Hence, since $\FTT$
preserves recoils, $\NFT$ preserves the descent composition
\begin{equation}
	\Des(\NFT(\sigma)) = \Des(\sigma).
\end{equation}
Moreover, $\inv(\sigma^{-1})=\inv(\sigma)$, so that
\begin{equation}
	\inv \NFT(\sigma) = \maj(\sigma^{-1}).
\end{equation}
Finally, Equation~(\ref{eq1}) becomes
\begin{equation}
\label{eq1b}
	\text{1 comes before $n$ in } \NFT(\sigma) \quad \Longleftrightarrow \quad
	\text{$n-1$ comes before $n$ in } \sigma.
\end{equation}

Now, let $k$ be the number of inversions of any element of $S$ and let us
consider that $1$ is before $n$ in any element of $S$, the other case being
similar.

Recall the set $\M_{k,+}$ of permutations $\sigma$ such that
$\maj(\sigma^{-1})=k$ and such that $n-1$ is before $n$.
Then this set is the disjoint union of the recoil classes $\C_{k,+}$ of the
compositions $I$ with major index $\maj(I)$ equal to $k$. Thanks to the
previous properties, the image of this set under $\NFT$ is exactly the set of
permutations with $k$ inversions and such that 1 is before $n$, that is,
the forgotten class $S$.

Since the bijection $\NFT$ preserves the descents of the permutations, we have
\begin{equation}
\sum_{\sigma\in S} F_{n,\Des(\sigma)}
= \sum_{\tau\in \NFT^{-1}(S)} F_{n,\Des(\tau)}
= \sum_{\tau\in \M_{k,+}} F_{n,\Des(\tau)}.
\end{equation}

But this sum is easy to compute: consider
\begin{equation}
	\sum_{\substack{\sigma\\ \Recoil(\sigma)=I}} \F_{\sigma},
\end{equation}
in the Hopf algebra $\FQSym$ of Free Quasi-Symmetric functions realized as
noncommutative polynomials in the free algebra. Then taking the commutative
image of this expression gives
\begin{equation}
	\sum_{\substack{\sigma\\ \Recoil(\sigma)=I}} F_{n,\Des(\sigma)}.
\end{equation}
But this sum is also equal to its realization in terms of the noncommutative
ribbon Schur function $\R_I$ whose commutative image is the usual ribbon Schur
function $r_I$. Hence
\begin{equation}
	\sum_{\substack{\sigma\\ \Recoil(\sigma)=I}} F_{n,\Des(\sigma)} = r_I.
\end{equation}
This shows indeed that 
\begin{equation}
	\sum_{\sigma\in S}F_{n,\Des(\sigma)} = \sum_{I\in \C_{k,+}} r_I
\end{equation}	
is a multiplicity-free sum of ribbon Schur functions.
\end{proof}

There is a very simple way to compute directly from a given
forgotten class the set of compositions appearing in its evaluation on
$\QSym$: consider the set of words defined by applying the inverses to 
$V$-permutations, that is permutations satisfying $\sigma_1>\sigma_2>\cdots
>\sigma_i<\sigma_{i+1}<\cdots<\sigma_n$ for some $1\le i\le n$. Those are
exactly the words such that each prefix ends by either its smallest letter or
its greatest letter. Then, by definition of the bijection $\FTT$, those
elements are fixed points under $\FTT$, so that the $V$-permutations are
fixed-points under the $\NFT$ bijection.
Now, since the $V$-permutations have all different recoil compositions,
the sum over the whole forgotten class of $F_{n,\Des(\sigma)}$ is equal to the
sum of the $r_I$, where $I$ are the recoil compositions of the
$V$-permutations of this forgotten class. Alternatively, conjugating by the 
Sch\"utzenberger involution $\#$ using Corollary~\ref{cor:Schutz},
the sum over the whole forgotten class of $F_{n,\Des(\sigma)}$ is equal to the
sum of the $r_I$, where $I$ runs over the set of reversed recoil compositions
of the $\Lambda$-permutations of this forgotten class.

\begin{example}
Let us compute the evaluation of the two forgotten classes of $S_8$ with $10$
inversions. The two canonical words are $12387654$ and $23458761$.

Using Example~\ref{ex-Leqs}, one can check that the $\Lambda$-permutations
belonging to the forgotten class of $12387654$ are
\begin{equation}
12387654,\ 12478653,\ 12568743,\ 13468752,\ 13567842.
\end{equation}
Their recoil composition are
$(41111)$, $(3212)$, $(3131)$, $(2321)$, $(224)$ respectively, so that
the evaluation of the forgotten class of $12387654$ is
\begin{equation}
r_{11114} + r_{2123} + r_{1313} + r_{1232} + r_{422}.
\end{equation}

One also checks that the $\Lambda$-permutations belonging to the forgotten
class of $23458761$ are
\begin{equation}
23458761, 23467851.
\end{equation}
Since their recoil composition are $(1511)$ and $(143)$ respectively, the
evaluation of the forgotten class of $23458761$ is
\begin{equation}
r_{1151} + r_{341}.
\end{equation}

One can easily check that the seven ribbons listed here exactly are the seven
compositions with major index equal to 10.
\end{example}


\end{document}